\documentclass[a4paper,12pt,fleqn]{article}
\usepackage[latin1]{inputenc}
\usepackage[T1]{fontenc}
\usepackage[a4paper,lmargin=4.1cm,textwidth=12.8cm,tmargin=4.9cm,textheight=18.5cm]{geometry}
\usepackage[all,2cell,v2]{xy}
\usepackage[all]{xy}
\UseAllTwocells
\usepackage[square,numbers,sort]{natbib}
\usepackage{ifthen}
\usepackage{mathbbol}
\usepackage{savesym}
\usepackage{amsmath}
\usepackage{amssymb}
\savesymbol{iint}
\restoresymbol{TXF}{iint}
\usepackage{mathptmx}
\usepackage[thmmarks,standard]{ntheorem}
\usepackage{microtype}
\usepackage{stmaryrd}
\usepackage[english]{babel}
\newboolean{PourEditeur}
\setboolean{PourEditeur}{false}
\ifthenelse{\boolean{PourEditeur}}{%
}{%
  \usepackage[colorlinks,linkcolor=blue,urlcolor=blue,citecolor=blue]{hyperref}
}
\usepackage{cleveref}
\theoremstyle{plain}
\crefformat{result}{result~#2#1#3}
\crefformat{proposition}{proposition~#2#1#3}
\crefformat{remark}{remark~#2#1#3}

\newcommand*{\Cat}{\ensuremath{\mathbb{C}\mathrm{at}}}

\newcommand*{\Mod}{\operatorname{\mathbb{M}\mathrm{od}}}

\DeclareMathOperator{\Mnd}{{\mathbb {M}\mathrm{nd}}}
\DeclareMathOperator{\Adj}{{\mathbb {A}\mathrm{dj}}}
\DeclareMathOperator{\Ar}{{\mathbb {A}\mathrm{r}}}

\title{Notes on $n$-Transformations by Theories ($n\in {\mathbb{N}^*}$)}
\author{Camell Kachour}
\begin{document}
\maketitle
\begin{abstract}
 In \cite{ber} Clemens Berger showed that weak categories of Michael Batanin \cite{bat} can be defined
 as model of a certain kind of theories that he called ``homogeneous theorie''. By using the work
 of Mark Weber on the Abstract Nerves \cite{web} for the
 specific case of the $n$-Transformations  \cite{k}, we show that we can
 also define Weak Functors, Weak Natural Transformations, and so on, as models of certain kind of
 colored theories which are homogeneous as well. 
\end{abstract}

\tableofcontents

In "pursuing stacks" \cite{gro} Alexander Grothendieck gave 
his definition of weak omega groupoids where he see
a weak omega groupoid as a model of some specific theories called "coherateurs" and a
slight modification of this definition led to a notion of weak omega category \cite{mal}.
 Thus in the spirit of Grothendieck, weak and higher structures can be seen as model
 of certain kind of theories. Later Clemens Berger in \cite{ber} gave a rigourous description of 
 the weak omega categories of Michael Batanin \cite{bat} in term of models for some kind of theories that he called "homogeneous theories", where he followed for the higher dimension, the idea of Boardman-Vogt whose 
 have shown in \cite{board}  that building an algebra for an operad is closely related to the construction of a model for some
 adapted homogeneous theory. One of the main idea of Clemens Berger was to build a Nerves for each
 weak omega category of Batanin as we do with the classical Nerve of a category \cite{mac} and he proved that
  the induced Nerve Functor is actually fully-faithfull as the classic one. He also characterised 
  presheaves living in the essential image of a such Nerve as some presheaves verifying some
  generalized Segal condition where the classical Segal condition is a particular case. 
  The main ingredient which allowed him to see weak categories as some specific presheaves 
  verifying  the Segal condition are the category of Trees of Michael Batanin \cite{bat}. Actually 
  in the formalism developped by Clemens Berger the classical case of the categories can be recovered with the
   category of the trees with level one considered as a subcategory of the category of graphs
     $\mathbb{G}r$  and which its free category is exactly the simplicial category $\Delta$. Suprisingly, this idea
  of Clemens Berger allow us also to see all these constructions as a generalisation of the process of
  building a Lawvere Theory. Indeed a Lawvere theory can be seen as a functor 
  which is an identity on the objects from the skelett of the category of the
  finite sets $\mathbb{N}$ to a small category and he extended this idea by considering 
  the category of graphical trees $\Theta_{0}$ of Michael Batanin instead of the category $\mathbb{N}$ to built his theories. Thus in the
  spirit of Clemens Berger we can see the category  $\Delta$ as a kind of Lawvere Theory, but also
  the  category $\Theta$ of Joyal, as an other kind of Lawvere Theory, etc.
  
  But the first step to be able to built a generalisation of  these kind of "generalized Lawvere  Theories" is to recognize the basic datas which allowed us to construct them and these basic datas should also produce at the same time a theory, an associated nerve, a notion of "Segal Condition" and a theorem which give a caracterisation of the essential image of the nerve in term of this "Segal Condition". These ideas are implicit in the work of Clemens Berger \cite{ber}
    but Mark Weber in \cite{web} make them explicitly by defining the notion of category with arities and the notion 
  of monad with arities. These notions allow us to built all the ingredients that we need to generalize
  the classical Segal condition for categories or the "globular Segal condition" for higher categories made by
  Clemens Berger in \cite{ber}. In this paper we are going to see that the formalism developped by Mark Weber and the level of generalisation of his Abstract Nerve \cref{NerveTheorem} allow us to applied it for the specific case of the $n$-Transformations (\cite{k}). Thus 
  we are going to show that $n$-Transformations are in fact models for some specific "colored theories". 
  However we avoid to prove that these theories are homogeneous in the sense of Clemens Berger (see \cite{ber}),
  even if it is easy to see that intuitively they are, because this property which was important for 
  Clemens Berger in his proofs are less important here thanks to the \cref{NerveTheorem} of Mark Weber. This work was exposed in the Australian Category Seminar in september 2010.

   I am grateful to my Supervisors Michael Batanin and Ross Street for giving me the chance to be their student, for their trust in me and for their huge effort to allow me to work with them in Sydney.
   I am also grateful to Dominic Verity for his patience and kidness especially when he tried to explain me some basic aspect about Lawvere Theories. Finally I am grateful to Mark Weber for explaining me his theory of the Abstract Nerves  which allowed me to see quickly that the monads of the $n$-Transformations are Parametric Right Adjoint.
      
   I dedicate this work to Ross Street.

  \vspace*{1cm}   
   
\section{The Categories $\Mnd$ and $\Adj$}
 This first paragraph is devoted to recall properly the category of monads $\Mnd$ and the category of adjunctions $\Adj$ that we need to built the Coglobular Complex of Kleisli Categories of the 
 $n$-Transformations ($n\in {\mathbb{N}^*}$).
 It is easy to notice, but still important, that these categories $\Mnd$ and $\Adj$  are slightly different from those which were defined in \cite{defop} and which allowed us to built the Globular Complex of Eilenberg-Moore Categories of the $n$-Transformations. We will see in \cref{TheorieNSC} that theories of the $n$-Transformations live precisely in this complex.

\label{MndAdj}
 We recall briefly a pretty but classical construction of the 
 adjoint pair $(\Adj, \Mnd, \mathbb{K}, \mathbb{U}, \eta, \epsilon)$ which is adapted
 to work with Kleisli categories. For this paper we need only the functor $\mathbb{K}$.
 
 \subsection{Definition.} 
 \label{Def}
Here is the category of monads $\Mnd$ adapted for Keisli categories:
\begin{itemize}
\item Its class of objects $\Mnd(0)$ are monads $(\mathcal{G},(T,\eta,\mu))$
\item Its class of morphisms $\Mnd(1)$
    \[\xymatrix{(\mathcal{G},(T,\eta,\mu))\ar[r]^{(Q,q)}&(\mathcal{G}',(T',\eta',\mu'))}\]
    
  are given by functors $\xymatrix{\mathcal{G}\ar[r]^{Q}&\mathcal{G}'}$  and natural transformations
   $\xymatrix{QT\ar[r]^{q}&T'Q}$ such that we have the two following axioms:
   \end{itemize}

\begin{enumerate}

  \item $\xymatrix{Q\ar[r]^{\eta'Q}\ar[d]_{Q\eta}&T'Q\\ QT\ar[ru]_{q} }$
  
  \item $\xymatrix{QT^{2}\ar[d]_{Q\mu}\ar[r]^{qT}&T'QT\ar[r]^{T'q}&T'^{2}Q\ar[d]^{\mu'Q} \\
 QT\ar[rr]_{q}&&T'Q}$
 
  \end{enumerate}
  
  The composition of two morphisms of monads
  
  \[\xymatrix{(\mathcal{G},(T,\eta,\mu))\ar[r]^{(Q,q)}&
  (\mathcal{G}',(T',\eta',\mu'))\ar[r]^{(Q',q')}&(\mathcal{G}",(T",\eta",\mu"))}\]
  
  is given by $(Q,q)\circ (Q',q'):=(Q'\circ Q,q'Q\circ Q'q)$ and the identity monad is given by
  $1_{(\mathcal{G},(T,\eta,\mu))}:=(1_{\mathcal{G}}, 1_{T})$
  \begin{remark}
 It is not difficult to see that $(Q,q)\circ (Q',q')$ is a morphism of  $\Mnd$ and that 
  this composition and this identity put on $\Mnd$ a structure of category.
  
   \end{remark}
   
   The category $\Adj$ has the following definition:
\begin{itemize}
\item Its class of objects $\Adj(0)$ are pairs of adjunction 
$(\mathcal{A}, \mathcal{G}, L, U, \eta, \epsilon)$ where $L$ is the left adjoint of $U$ 
 \[\xymatrix{\mathcal{A}\ar@<.5 ex>[r]^{U}&\mathcal{G}\ar@<.5 ex>[l]^{L}}\]

 \item Morphisms of $\Adj(1)$ 
     \[\xymatrix{(\mathcal{A}, \mathcal{G}, L, U, \eta, \epsilon)\ar[r]^(.47){(P,Q)}&
    (\mathcal{A}', \mathcal{G}', L', U', \eta', \epsilon')}\]
    
  are given by functors $\xymatrix{\mathcal{A}\ar[r]^{P}&\mathcal{A}'}$  and 
  $\xymatrix{\mathcal{G}\ar[r]^{Q}&\mathcal{G}'}$
    such that the following square commute 
    \[\xymatrix{\mathcal{A}\ar[r]^{P}&\mathcal{A}'\\
        \mathcal{G}\ar[r]_{Q}\ar[u]^{L}&\mathcal{G}'\ar[u]_{L'}}\]

       \end{itemize}

 \subsection{Functors between $\Adj$ and $\Mnd$}  
  \label{FunctMndAdj}
  
   We have a functor
     \[\xymatrix{\Adj\ar[r]^{\mathbb{U}}&\Mnd}\]
  
 which send  the object $(\mathcal{A}, \mathcal{G}, L, U, \eta, \epsilon)$ of  $\Adj$ to the
 object $(\mathcal{G}, (UL, \eta, U\epsilon L))$ of $\Mnd$ and which send the morphism $(P,Q)$ of
   $\Adj$ to the morphism $(Q,q)$ of $\Mnd$ such that 
   $q=(U'P\epsilon \circ^{2}_{1} \eta'QU)L$ and it is
   not difficult to see that this morphism lives in $\Mnd$
   
  We also have the functor
  
     \[\xymatrix{\Mnd\ar[r]^{\mathbb{K}}&\Adj}\]
     
   which send the monad $(\mathcal{G},(T,\eta,\mu))$ to the adjunction 
   $(Kl(T), \mathcal{G}, L_{T}, U_{T}, \eta_{T}, \epsilon_{T})$ coming from the Kleisli construction.
   Objects of $Kl(T)$ are objects of  $\mathcal{G}$  and morphisms 
   $\xymatrix{G\ar[r]^{f}&G'}$ of $Kl(T)$ are given by morphisms $\xymatrix{G\ar[r]^{f}&T(G')}$
   of  $\mathcal{G}$. Also if  $\xymatrix{G\ar[r]^{g}&G'}$  live in  $\mathcal{G}$ then 
   $L_{T}(g)=\eta(G')\circ g$ and if  $\xymatrix{G\ar[r]^{f}&G'}$ live in $Kl(T)$ then 
   $U_{T}(f)=\mu(G')\circ T(f)$. Finally $\mathbb{K}$ send the morphism $(Q,q)$ of $\Mnd$
   to the morphism $(P,Q)$ of $\Adj$ such that if 
     $\xymatrix{G\ar[r]^{f}&G'}$ is a morphism of $Kl(T)$ then $P(f)=q(G')Q(f)$ 
     
    We can prove without difficulty that $\mathbb{U}\circ \mathbb{K}=1_{\Mnd}$ and also that
    $\mathbb{K}$ is left adjoint to $\mathbb{U}$       
    
 \section{Coglobular Complex of Kleisli of the $n$-Transformations ($n\in {\mathbb{N}^*}$).}
   \label{CogloKleisli}
   
   This second paragraph is just an application of the first paragraph for the specific case of the $n$-Transformations. In this
 paragraph we build explicitly the Coglobular Complex of Kleisli of the $n$-Transformations ($n\in {\mathbb{N}^*}$).   
    
 Consider the coglobular complex of $CT-\Cat_{c}$ of the globular contractible colored operads
 of the $n$-Transformations \cite{k} 
 
 \[\xymatrix{B^{0}\ar[rr]<+2pt>^{\delta^{0}_{1}}\ar[rr]<-2pt>_{\kappa^{0}_{1}}
  &&B^{1}\ar[rr]<+2pt>^{\delta^{1}_{2}}\ar[rr]<-2pt>_{\kappa^{1}_{2}}
  &&B^{2}\ar@{.>}[r]<+2pt>^{}\ar@{.>}[r]<-2pt>_{}
  &B^{n-1}\ar[rr]<+2pt>^{\delta^{n-1}_{n}}\ar[rr]<-2pt>_{\kappa^{n-1}_{n}}
  &&B^{n}\ar@{.}[r]<+2pt>\ar@{.}[r]<-2pt>&}\]

  For each $j\in \mathbb{N}$ we note $(\underline{B}^{j}, \mu^{j}, \eta^{j})$ the corresponding monads. 
  
  Given the following functors "choice of a color" for each $j\in{\{1,2\}}$
  $\xymatrix{\omega-\mathbb{G}r\ar[r]^(.4){{i_{j}}_{\ast}}&\omega-\mathbb{G}r/1\cup 2}$ which send the 
  $\omega$-graph $G$ to the bicolored $\omega$-graph  $i_{j}\circ !_{G}$ and which send a morphism 
  $f$ to $f$. It result from the morphisms of color $\xymatrix{1\ar[r]^(.42){i_{j}}&1\cup 2}$ (see \cite{k}).  
  
  By definition of the monads $\underline{B}^{0}$ and  $\underline{B}^{1}$ we have the following natural 
  transformation
       
      \[\xymatrix{%
     \omega-\mathbb{G}r\ar[d]_{{i_{1}}_{\ast}}\ar[r]^-{\underline{B}^{0}}
     &\omega-\mathbb{G}r\ar[d]^{{i_{1}}_{\ast}} 
     \\
     \omega-\mathbb{G}r/1\cup 2\ar[r]_-{\underline{B}^{1}}
     &*+!L(.4){\omega-\mathbb{G}r/1\cup 2}
       \save+U+/l.5pc/*\cir<1.25pc>{l^d},+/l1.23pc/,*=0^@{<},+/u1.25pc/*!R(.5){\labelstyle \delta^{0}_{1}}\restore
    }\]
          
 and also we have the following natural transformation
  
   \[\xymatrix{%
    \omega-\mathbb{G}r\ar[d]_{{i_{2}}_{\ast}}\ar[r]^-{\underline{B}^{0}}&    
     \omega-\mathbb{G}r\ar[d]^{{i_{2}}_{\ast}} 
     \\
     \omega-\mathbb{G}r/1\cup 2\ar[r]_-{\underline{B}^{1}}
     &*+!L(.4){\omega-\mathbb{G}r/1\cup 2}
      \save+U+/l.5pc/*\cir<1.25pc>{l^d},+/l1.23pc/,*=0^@{<},+/u1.25pc/*!R(.5){\labelstyle \kappa^{0}_{1}}\restore
     }\]
     
   Furthemore we have for each $j\geqslant 1$ the following natural transformations     
     
        \[\xymatrix{%
     \omega-\mathbb{G}r/1\cup 2\ar[d]_{1}\ar[r]^-{\underline{B}^{j}}
     &*+!L(.4){\omega-\mathbb{G}r/1\cup 2} \ar[d]^{1} 
     \\
     \omega-\mathbb{G}r/1\cup 2\ar[r]_-{\underline{B}^{j+1}}
     &*+!L(.4){\omega-\mathbb{G}r/1\cup 2}
       \save+U+/l.5pc/*\cir<1.25pc>{l^d},+/l1.23pc/,*=0^@{<},+/u1.25pc/*!R(.5){\labelstyle \delta^{j}_{j+1}}\restore
    }\]

    \[\xymatrix{%
    \omega-\mathbb{G}r/1\cup 2 \ar[d]_{1}\ar[r]^-{\underline{B}^{j}}
    &*+!L(.4){\omega-\mathbb{G}r/1\cup 2} \ar[d]^{1} 
     \\
     \omega-\mathbb{G}r/1\cup 2\ar[r]_-{\underline{B}^{j+1}}
     &*+!L(.4){\omega-\mathbb{G}r/1\cup 2}
       \save+U+/l.5pc/*\cir<1.25pc>{l^d},+/l1.23pc/,*=0^@{<},+/u1.25pc/*!R(.5){\labelstyle \kappa^{j}_{j+1}}\restore
     }\]
     
   and it is easy to see that these natural transformations fit well the axioms of 
   morphisms of $\Mnd$ (And it is similar to the construction in \cite{k}).
   The functoriality of the building a monad from a $\mathbb{T}$-Category implied that we can build
   the corresponding coglobular complex of $\Mnd$ 
   
   \[\xymatrix{(\omega-\mathbb{G}r, \underline{B}^{0})
    \ar[r]<+2pt>^(.45){({i_{1}}_{\ast}, \delta^{0}_{1})}
    \ar[r]<-2pt>_(.45){({i_{2}}_{\ast}, \kappa^{0}_{1})}& 
    (\omega-\mathbb{G}r/1\cup 2,\underline{B}^{1})
   \ar@{.>}[r]<+2pt>^{}\ar@{.>}[r]<-2pt>_{} &
   (\omega-\mathbb{G}r/1\cup 2,\underline{B}^{n}) 
   \ar[rr]<+2pt>^(.75){(1_{\omega-\mathbb{G}r/1\cup 2}, \delta^{n}_{n+1})}
	   \ar[rr]<-2pt>_(.75){(1_{\omega-\mathbb{G}r/1\cup 2}, \kappa^{n}_{n+1})}&&\ar@{.}[r]&}\]

   If $\xymatrix{\Adj\ar[r]^{\mathbb{P}}&\Cat}$ is the projection functor,  then the 
   functor
     \[\xymatrix{\Mnd\ar[r]^{\mathbb{K}}&\Adj\ar[r]^{\mathbb{P}}&\Cat}\]     
     
     brings to light  the following coglobular complex of Kleisli of the $n$-Transformations ($n\in {\mathbb{N}^*}$)  
     
     \[\xymatrix{Kl(\underline{B}^{0})\ar[rr]<+2pt>^{\delta^{0}_{1}}\ar[rr]<-2pt>_{\kappa^{0}_{1}}
  &&Kl(\underline{B}^{1})\ar[rr]<+2pt>^{\delta^{1}_{2}}\ar[rr]<-2pt>_{\kappa^{1}_{2}}
  &&Kl(\underline{B}^{2})\ar@{.>}[r]<+2pt>^{}\ar@{.>}[r]<-2pt>_{}
  &Kl(\underline{B}^{n-1})\ar[rr]<+2pt>^{\delta^{n-1}_{n}}\ar[rr]<-2pt>_{\kappa^{n-1}_{n}}
  &&Kl(\underline{B}^{n})\ar@{.}[r]<+2pt>\ar@{.}[r]<-2pt>&}\]   
    
 \section{The Category $\Mnd\Ar$ of Monads with Arities.} 
  \label{MndAr}
  
  This third paragraph recall the basic tools of the Nerve theory as developped by Mark Weber in \cite{web}. The
 only originality of this paragraph is to define properly two categories, $\Ar\Mnd$ and    
 $\Mnd\Ar$, respectively the category of Categories with Arities equipped with Monads and 
  the category of Monads with Arities. We hope that the explicit description of these categories make easier the construction of the Coglobular Complex of the Theories of the $n$-Transformations of the fourth paragraph.
 But overall we point out the case of the Monads Parametric Right Adjoints (Monads p.r.a, for shorter) which are the most important example of monads with arities. Actually we see in \cref{thMonadeNSCpra} that monads of the $n$-Transformations are indeed p.r.a, thus they allow us to show quickly that the $n$-Transformations are models for their corresponding theories.

    \subsection{The Category $\Ar$ of Categories with Arities.}  
   \label{Ar}
     
  The definition of the category $\Ar$ of Categories with Arities can also be found in \cite{mel}.
  
  \begin{description}
  \item[] Its objects are triple $(\Theta_{0}, i_{0}, \mathcal{A})$ where 
  $\xymatrix{\Theta_{0}\ar[r]^{ i_{0}}&\mathcal{A}}$ is a functor such that 
  
  \begin{itemize}
  \item $\mathcal{A}$ is cocomplet
  \item $i_{0}$ is fully faithfull.
  \item $\forall a\in{\mathcal{A}(0)}$, the presheaf 
  $hom_{\mathcal{A}}(i_{0}(-),a)$ which lives in $\widehat{\Theta_{0}}$ is fully-faithfull.
\end{itemize}

 This last condition is expressed by saying that $i_{0}$ is dense.
  
  \item[ ] Morphisms of $\Ar$ 
  $\xymatrix{(\Theta_{0}, i_{0}, \mathcal{A})\ar[r]^{(F,l)}&(\Theta_{1}, i_{1}, \mathcal{B})}$ are given by two functors
  $F$ and $l$ such that 
  
  \begin{itemize}
  \item  The following diagram is commutative
     \[\xymatrix{\Theta_{1} \ar[r]^{ i_{1}}&\mathcal{B}\\
    \Theta_{0} \ar[u]^{l}\ar[r]_{ i_{1}}&\mathcal{A}\ar[u]_{F}}\]
    
  \item They satisfy the Beck-Chevalley condition which means that if we note
  $\xymatrix{1_{\mathcal{A}}\ar[r]^{\eta^{0}}&{i_{o}}_{\ast} {i_{0}}^{\ast}}$ 
 and   $\xymatrix{i_{1}^{\ast} {i_{1}}_{\ast}\ar[r]^{\epsilon^{1}}&1_{\widehat{\mathcal{B}}}}$ respectivly
 the unit and the counit of the adjunctions $ i_{0}^{\ast}\dashv {i_{o}}_{\ast}$ and $ {i_{1}}^{\ast}\dashv {i_{1}}_{\ast}$ which result from the commutative square of presheaves categories
    
       \[\xymatrix{\widehat{\Theta_{1}} \ar[d]_{l^{\ast}}&\widehat{\mathcal{B}}
          \ar[l]_{ i_{1}^{\ast}}\ar[d]^{F^{\ast}}\\
              \widehat{\Theta_{0}}&\widehat{\mathcal{A}}\ar[l]^{{i_{0}}^{\ast}}}\]
 then the mate 
 
  \[\xymatrix{F^{\ast} {i_{1}}_{\ast}\ar[rrr]^
  {\Psi=({i_{0}}_{\ast}l^{\ast}\epsilon^{1})\circ^{2}_{1}(\eta^{0}F^{\ast} {i_{1}}_{\ast})}&&&{i_{0}}_{\ast}l^{\ast}}\] 
 
 associated to the natural identity 
 $\xymatrix{i_{0}^{\ast}F^{\ast}\ar[r]^{1}&l^{\ast}{i_{1}}^{\ast}}$, need to be a natural isomorphism.
  
  \end{itemize}
 
\end{description}
  
  \begin{remark}
  R.Street and M.Kelly showed in \cite{kelstr}  that building the mate of a natural transformation is functorial. Thus morphisms compose easily.
  
   \end{remark}

    \subsection{The Category $\Ar\Mnd$ of Categories with Arities equipped with Monads.} 
   \label{ArMnd}
   
     Lets note $\Ar\Mnd$ the following category
  \begin{enumerate}
  \item An object is given by $((\Theta_{0}, i_{0}, \mathcal{A}), (T, \eta, \mu))$ where $(\Theta_{0}, i_{0}, \mathcal{A})$
  is an object of $\Ar$  and $( \mathcal{A},(T, \eta, \mu))$ is a monad.  
  \item An arrow of $\Ar\Mnd$  
      \[\xymatrix{((\Theta_{0}, i_{0}, \mathcal{A}), (T, \eta, \mu))
      \ar[r]^{(Q,q,l)}&((\Theta_{1}, i_{1}, \mathcal{B}), (T', \eta', \mu'))}\]
      is given by an arrow of $\Mnd$ 
       \[\xymatrix{( \mathcal{A},(T, \eta, \mu))\ar[r]^{(Q,q)}&(\mathcal{B}, (T', \eta', \mu'))}\] 
       and an arrow of $\Ar$
          \[\xymatrix{(\Theta_{0}, i_{0}, \mathcal{A})\ar[r]^{(Q,q)}&(\Theta_{1}, i_{1}, \mathcal{B})}\]  
\end{enumerate}  

 For an object  $((\Theta_{0}, i_{0}, \mathcal{A}), (T, \eta, \mu))$ in $\Ar\Mnd$ 
  we can build the following four tools
  
  \begin{itemize}
  
    \item The functor $\xymatrix{\mathcal{A}\ar[rr]^{hom_{\mathcal{A}}(i_{0}(-),-)}&&\widehat{\Theta_{0}}}$ which gives the
    presheaf $hom_{\mathcal{A}}(i_{0}(-),a)$ for all $a\in \mathcal{A}$  
    \item The theory $\Theta_{T}$
     \label{Theory}    
    \[\xymatrix{\Theta_{0}\ar[rd]_{j}\ar[r]^{i_{0}}&\mathcal{A}\ar[r]^(.47){L^{T}}&T-\mathbb{A}lg\\
   &\Theta_{T}\ar[ru]_{i}}\]    
    
  that we obtain with the factorisation of the functor $L^{T}i_{0}$ by  
  a functor $j$ which is an identity on the objects and a fully-faithfull functor $i$. 
  We will see later \cref{TheorieNSC} that for
  the specific case of the $n$-Transformations, $\Theta_{T}$ can be defined in a elegant way
    as a full subcategory of the Kleisli category $Kl(T)$   
   
   \item The Nerve $N_{T}$
    \label{Nerve}   
   It is the functor 
    \[\xymatrix{T-\mathbb{A}lg\ar[rrrr]^{N_{T}:=hom_{T-\mathbb{A}lg}(i(-),-)}&&&&\widehat{\Theta_{T}}}\]
    
  The classical nerve and the nerve of Clemens Berger \cite{ber} are special case of this construction
   
    \item The restriction functor $res_{j}$
    
    It is the functor     
    \[\xymatrix{\widehat{\Theta_{T}}\ar[rrrr]^{res_{j}:=hom_{\widehat{\Theta_{T}}}(Y\circ j(-),-)}&&&&\widehat{\Theta_{0}}}\]
    
   where $Y$ is the Yoneda embeded $\xymatrix{\Theta_{T}\ar[r]^{Y}&\widehat{\Theta_{T}}}$

  This functor allow us to have a good notion of Generalised Segal Condition.    
    
  \end{itemize}
  
  \begin{definition}[Generalised Segal Condition]
   \label{Segal}  
  Given an object $((\Theta_{0}, i_{0}, \mathcal{A}), (T, \eta, \mu))$ in $\Ar\Mnd$.
   A presheaf $Z\in {\widehat{\Theta_{T}}}$ satisfy the Segal Condition if $res_{j}(Z)$ belongs to the image of 
    $hom_{\mathcal{A}}(i_{0}(-),-)$. 
 \end{definition} 
  
 \subsection{The Category $\Mnd\Ar$ of Monads with Arities and the Nerves Theorem.} 
  \label{MndArNerve} 
 $\Mnd\Ar$ is a full subcategory of  $\Ar\Mnd$ such that each of its 
 object $((\Theta_{0}, i_{0}, \mathcal{A}), (T, \eta, \mu))$ has in more the two following properties
 
 \begin{itemize}
  \item $T=Lan_{i_{o}}(T\circ i_{o})$
  \item $hom_{\mathcal{A}}(i_{o}(-), Lan_{i_{o}}(T\circ i_{o})(-))=
  Lan_{i_{o}}(hom_{\mathcal{A}}(i_{o}(-),T\circ i_{o})(-))$
\end{itemize} 

\begin{theorem}[Nerve Theorem]
 \label{NerveTheorem}
 If $((\Theta_{0}, i_{0}, \mathcal{A}), (T, \eta, \mu))$ belong to $\Mnd\Ar$ then 
 \begin{itemize}
  \item $N_{T}$ is fully faithfull
  \item 
  
  \begin{minipage}[c]{0.4\linewidth}
       $Z\in \widehat{\Theta_{T}}$ satisfy the Segal Condition 
           \end{minipage}
    \quad$\Longleftrightarrow$\quad
    \begin{minipage}[c]{0.4\linewidth}
      There is a $(G,v)\in T-\mathbb{A}lg$ such that $N_{T}(G,v)=Z$
          \end{minipage}  
   \end{itemize}
 \end{theorem}  
 
 In the following we will note $\Mod(\Theta_{T})$ the essential image of $N_{T}$.
 
 \subsection{The fundamental example of Monads with Arities: The Monads Parametric Right Adjoints (Monads p.r.a).}  
  \label{FundamentalExample} 
 Given $\mathcal{A}$ a category with a final object $1$, and a functor 
 $\xymatrix{\mathcal{A}\ar[r]^{F}&\mathcal{B}}$
 
 We have the following factorisation:
 
    \[\xymatrix{\mathcal{A}\ar[rd]_{F_{1}}\ar[rr]^{F}&&\mathcal{B}\\
    &\mathcal{B}/F(1)\ar[ru]_{cod}}\] 
    
  where $F_{1}(a):=F(!_{a})$. In that case we have the following important definition
 \begin{definition}[Street 2001]
  The last $F$ is qualified as Parametric Right Adjoints (p.r.a) if $F_{1}$ has a left adjoint.
\end{definition}
\begin{definition}
A monad $(\mathcal{G},(T,\eta,\mu))$ is p.r.a if $T$ is p.r.a. and if its unit and multiplication are cartesian.
 \end{definition}  
 
 \begin{theorem}[Weber, \cite{web}]
  \label{TheoremPRA}
 Given $(\Theta_{0}, i_{0}, \mathcal{A})\in \Ar$. If a monad $(\mathcal{A},(T,\eta,\mu))$ is p.r.a
then $((\Theta_{0}, i_{0}, \mathcal{A}), (T,\eta,\mu))\in \Mnd\Ar$
\end{theorem}   
The two following propositions will be useful even if they are easy
 \begin{proposition}[Weber, \cite{web}]
  \label{PropoCart}
  Given a cartesian transformation $\xymatrix{F\ar[r]^{\gamma}&G}$ between two cartesian functors $F$ and $G$.
If $G$ is p.r.a the $F$ is p.r.a as well.
\end{proposition}  
 \begin{proposition}
  \label{PropoDomain}
 Given a category $\mathcal{A}$ and an object $A\in \mathcal{A}$. Then the functor domain
               \[\xymatrix{\mathcal{A}/A\ar[r]^{U_{A}}&\mathcal{A}}\]
  is p.r.a
\end{proposition}   
   
  \section{Coglobular Complex of the Theories of the $n$-Transformations($n\in {\mathbb{N}^*}$).}
    \label{CogloTheorie}
    
    This fourth paragraph is an application of the last paragraph for the $n$-Transformations. In particular we built the 
 Coglobular Complex of the Theories of the $n$-Transformations. First of all we exhibit Categories of Arities for the $n$-Transformations where we can immediately see their colored nature. Then we construct  the Theories of the $n$-Transformations where in particular we can see their bicolored features and then we describe these colored theories as full 
 subcategories of their Kleisli categories. Finally we exhibit the Coglobular Complex of the Theories of the $n$-Transformations. 
  \subsection{Categories of Arities for the $n$-Transformations.} 
     \label{ArNSC}
   Given $\Theta_{0}$ the graphic trees category of [Batanin, Berger, Joyal]. We have the following proposition
  \begin{proposition}
   \label{PropoArNSC}  
  For all $n\in \mathbb{N}^{\ast}$ the following canonical inclusion functors 
  \[\xymatrix{\Theta_{0}\sqcup ... \sqcup \Theta_{0} 
  \ar@{^{(}->}[r]^(.4){i_{0}}&\omega-\mathbb{G}r/1\cup 2 \cup ... \cup n} \]
  
 produce categories with arities. 
 \end{proposition} 
   
 \begin{proof}
 By definition of $\Theta_{0}$ which is a full subcategory of $\omega-\mathbb{G}r$ the inclusions
 $i_{0}$ are fully faithfull.
 
 Also for each $n\in{\mathbb{N}}$ we have the following diagram
 
     \[\xymatrix{
 *++{ \Theta_{0} \sqcup ... \sqcup \Theta_{0}}
 \ar@{>->}[d]_{Y}\ar@{^{(}->}[rr]^(.5){i_{0}}
  &&\omega-\mathbb{G}r/1\cup ... \cup n
 \ar[lld]^*!R(.75){\labelstyle\operatorname{hom}_{\omega-\mathbb{G}r/1\cup ... \cup n}(i_{0}(-),-)}
 \\
 \widehat{\Theta_{0} \sqcup ... \sqcup \Theta_{0}} 
 }\]
 
 and if $G, G' \in \omega-\mathbb{G}r/1\cup 2 \cup ... \cup n$, by 
 the lemma of Yoneda we get the following bijection
 
 \[hom_{\omega-\mathbb{G}r/1\cup 2 \cup ... \cup n}(G,G')
 \simeq \\
 hom_{\widehat{\Theta_{0} \sqcup\Theta_{0}\sqcup ... \sqcup \Theta_{0}}}
 (hom_{\omega-\mathbb{G}r/1\cup 2 \cup ... \cup n}(i_{0}(-),G),
 hom_{\omega-\mathbb{G}r/1\cup 2 \cup ... \cup n}(i_{0}(-),G'))\]
   \end{proof}
   
  For the $n$-Transformations the two following morphisms of $\Ar$ are
  important
  
   \[\xymatrix{\Theta_{0}\ar@{^{(}->}[d]_{i_{0}}\ar@< 2pt>[rr]^{{i_{1}}_{\ast}}
   \ar@< -2pt>[rr]_{{i_{2}}_{\ast}}&&
   \Theta_{0} \sqcup\Theta_{0}\ar@{^{(}->}[d]^{i_{0}}\\
  \omega-\mathbb{G}r
   \ar@< 2pt>[rr]^{{i_{1}}_{\ast}}
   \ar@< -2pt>[rr]_{{i_{2}}_{\ast}}&&\omega-\mathbb{G}r/1\cup 2 }\]
   
   where ${i_{1}}_{\ast}$ and ${i_{2}}_{\ast}$ are the functors "choice of a color" (See the \cref{CogloKleisli}).
   
\subsection{Theories of the $n$-Transformations($n\in {\mathbb{N}^*}$).}  
 \label{TheorieNSC}
Let us consider the case of the Categories with Arities equipped with Monads of the $n$-Transformations
$(( \Theta_{0},i_{0},\omega-\mathbb{G}r),(\underline{B^{0}}, \eta^{0}, \mu^{0}))$ and
$(( \Theta_{0} \sqcup\Theta_{0},i_{0},\omega-\mathbb{G}r/1\cup 2),(\underline{B^{i}}, \eta^{i}, \mu^{i}))$
 if $i\geqslant 1$
 
 We have the following factorisation
 
   \[\xymatrix{\Theta_{0}\ar[rd]_{j}\ar[r]^{i_{0}}&\omega-\mathbb{G}r\ar[r]^(.47){L^{0}}&\underline{B^{0}}-\mathbb{A}lg\\
   &\Theta_{\underline{B^{0}}}\ar[ru]_{i}}\] 
   
  and for each $i\geqslant 1$ we have the following factorisations
  
   \[\xymatrix{\Theta_{0} \sqcup\Theta_{0}\ar[rd]_{j}\ar[r]^(.4){i_{0}}&
   \omega-\mathbb{G}r/1\cup 2\ar[r]^(.58){L^{i}}&\underline{B^{i}}-\mathbb{A}lg\\
   &\Theta_{\underline{B^{i}}}\ar[ru]_{i}}\] 
   
   where the functors $j$ are identity on the objects and the functors $i$ are fully faithfull. 
   The categories $\Theta_{\underline{B^{0}}}$, $\Theta_{\underline{B^{1}}}$, ...,$\Theta_{\underline{B^{i}}}$, ...
   are the theories of the $n$-Transformations (by abuse we call $\Theta_{\underline{B^{0}}}$
   the theory of the $0$-Transformations, which is actually the theory built by Clemens Berger in \cite{ber}). We can also give them the following alternative definition: Each
   $\Theta_{\underline{B^{i}}}$ can be seen as the full subcategory of the Kleisli category 
   $Kl(\Theta_{\underline{B^{i}}})$ 
   (see the paragraph \cref{CogloKleisli}) 
   which objects are the bicolored trees if $i\geqslant 1$ (i.e belong
   in $\Theta_{0} \sqcup\Theta_{0}$), and which objects are the trees if $i=0$. With this description
   we obtain the Coglobular Complex of the theories of the $n$-Transformations which is seen as a subcomplex
   of the Coglobular Complex  of the Kleisli categories of the $n$-Transformations

    \[\xymatrix{\Theta_{\underline{B^{0}}}\ar@{^{(}->}[d]\ar[rr]<+2pt>^{\delta^{0}_{1}}\ar[rr]<-2pt>_{\kappa^{0}_{1}}
  &&\Theta_{\underline{B^{1}}}\ar@{^{(}->}[d]\ar[rr]<+2pt>^{\delta^{1}_{2}}\ar[rr]<-2pt>_{\kappa^{1}_{2}}
  &&\Theta_{\underline{B^{2}}}\ar@{^{(}->}[d]\ar@{.>}[r]<+2pt>^{}\ar@{.>}[r]<-2pt>_{}
  &\Theta_{\underline{B^{n-1}}}\ar@{^{(}->}[d]\ar[rr]<+2pt>^{\delta^{n-1}_{n}}\ar[rr]<-2pt>_{\kappa^{n-1}_{n}}
  &&\Theta_{\underline{B^{n}}}\ar@{^{(}->}[d]\ar@{.}[r]<+2pt>\ar@{.}[r]<-2pt>&\\
  Kl(\underline{B}^{0})\ar[rr]<+2pt>^{\delta^{0}_{1}}\ar[rr]<-2pt>_{\kappa^{0}_{1}}
  &&Kl(\underline{B}^{1})\ar[rr]<+2pt>^{\delta^{1}_{2}}\ar[rr]<-2pt>_{\kappa^{1}_{2}}
  &&Kl(\underline{B}^{2})\ar@{.>}[r]<+2pt>^{}\ar@{.>}[r]<-2pt>_{}
  &Kl(\underline{B}^{n-1})\ar[rr]<+2pt>^{\delta^{n-1}_{n}}\ar[rr]<-2pt>_{\kappa^{n-1}_{n}}
  &&Kl(\underline{B}^{n})\ar@{.}[r]<+2pt>\ar@{.}[r]<-2pt>&}\]  
   
  \section{Final Bouquet: Coglobular Complex in $\Mnd\Ar$ of the $n$-Transformations($n\in {\mathbb{N}^*}$).} 
   \label{FinalBouquet}  
    This fifth paragraph is a kind of "Final Bouquet" of some kind of coglobular and globular complex construction of the $n$-Transformations. Indeed we show briefly that the monads of the $n$-Transformations are p.r.a which allow us to exhibit the Coglobular Complex in $\Mnd\Ar$ of the $n$-Transformations and also the Globular Complex of Nerves of the $n$-Transformations, and finally the equivalence in $\mathbb{G}lob(\Cat)$ which express the definition of the $n$-Transformations by theories, which is the outcome of this article.
   
  It is well known that 
  $(\omega-\mathbb{G}r, (\underline{B^{0}}, \eta^{0}, \mu^{0}))$ is a monad p.r.a 
  \cite{web}. In fact we 
  are going to see that all monads of the $n$-Transformations ($n\in {\mathbb{N}^*}$) have this property
  
 \begin{theorem}
  \label{thMonadeNSCpra} 
For all $i\geqslant 1$ the monad 
$(\omega-\mathbb{G}r/1\bigcup 2, (\underline{B^{i}}, \eta^{i}, \mu^{i}))$ is p.r.a
 \end{theorem} 
 \begin{proof}
 
 By definition of the monads $(\omega-\mathbb{G}r/1\bigcup 2, (\underline{B^{i}}, \eta^{i}, \mu^{i}))$ we have
 the following natural transformation (see \cite[propo 6.2.1 p 153]{lein} for its construction)
 
  \[\xymatrix{&\omega-\mathbb{G}r/1\bigcup 2\ar[rd]^{U_{1\bigcup 2}}&\\
   \omega-\mathbb{G}r/1\bigcup 2 \ar[ru]^{\underline{B^{i}}}\ar[rd]_{U_{1\bigcup 2}}&&\omega-\mathbb{G}r\\
   &\omega-\mathbb{G}r \ar[ru]_{\mathbb{T}}&}\]   
   
 where $U_{1\bigcup 2}$ is the functor domain of the \cref{PropoDomain} and $\mathbb{T}$ is the monad
   of the strict $\omega$-categories. It is easy to see that $U_{1\bigcup 2}$ is cartesian and p.r.a. The monad 
    $\mathbb{T}$ is p.r.a as well \cite{str}. Thus the \cref{PropoCart} show that 
    the functor $U_{1\bigcup 2}\underline{B^{i}}$ is p.r.a. But then $\underline{B^{i}}$ is also p.r.a  because with
    the factorisation of $\underline{B^{i}}$ and  $U_{1\bigcup 2}\underline{B^{i}}$  
 
          \[\xymatrix{\omega-\mathbb{G}r/1\bigcup 2
          \ar[rd]_{{\underline{B^{i}}}_{1}}\ar[rr]^{\underline{B^{i}}}&&\mathcal{B}\\
    &(\omega-\mathbb{G}r/1\bigcup 2)/\underline{B^{i}}(1_{1\bigcup 2})\ar[ru]_{cod}}\]

  \[\xymatrix{\omega-\mathbb{G}r/1\bigcup 2
          \ar[rd]_{({U_{1\bigcup 2}\underline{B^{i}}})_{1}}\ar[rr]^{U_{1\bigcup 2}\underline{B^{i}}}&&\omega-\mathbb{G}r\\
    &\omega-\mathbb{G}r/U_{1\bigcup 2}\underline{B^{i}}(1_{1\bigcup 2})\ar[ru]_{cod}}\]  
    
   we can directly see that  ${\underline{B^{i}}}_{1}=({U_{1\bigcup 2}\underline{B^{i}}})_{1}$
  \end{proof}
  
  Thus the objects  of  $\Ar\Mnd$:
  $((\Theta_{0} \sqcup\Theta_{0}, i_{0},\omega-\mathbb{G}r/1\cup 2),(\underline{B^{i}}, \eta^{i}, \mu^{i}))$
 $(i\geqslant 1)$ are more precisely objects of $\Mnd\Ar$. So we obtain the coglobular complex in   
  $\Mnd\Ar$ of the $n$-Transformations
  
   \begin{multline*}
   \xymatrix{((\omega-\mathbb{G}r, i_{0},\Theta_{0}),(\underline{B^{0}}, \eta^{0}, \mu^{0}))
   \ar[r]<+2pt>^-{\delta^{0}_{1}}
    \ar[r]<-2pt>_-{\kappa^{0}_{1}}& 
 ((\omega-\mathbb{G}r/1\cup 2 , i_{0}, \Theta_{0}\sqcup\Theta_{0}),(\underline{B^{1}}, \eta^{1}, \mu^{1})) 
 \ar[r]<+2pt>^-{\delta^{1}_{2}}\ar[r]<-2pt>_(.85){\kappa^{1}_{2}}&...}\\
\xymatrix{((\omega-\mathbb{G}r/1\cup 2 , i_{0}, \Theta_{0} \sqcup\Theta_{0}),(\underline{B^{i}}, \eta^{i}, \mu^{i}))    
   \ar[r]<+2pt>^-{\delta^{i}_{i+1}}
	   \ar[r]<-2pt>_-{\kappa^{i}_{i+1}}&...}
 \end{multline*}
  which brings to light  the Globular Complex of Nerves of the $n$-Transformations
  
  \[\xymatrix{\ar@{.>}[r]<+2pt>^{}\ar@{.>}[r]<-2pt>_{}
  &\underline{B^{n}}-\mathbb{A}lg\ar[d]_{N_{\underline{B^{n}}}}\ar[r]<+2pt>^{\sigma^{n}_{n-1}}\ar[r]<-2pt>_{\beta^{n}_{n-1}}
  &\underline{B^{n-1}}-\mathbb{A}lg\ar[d]^{N_{\underline{B^{n-1}}}}\ar@{.>}[r]<+2pt>^{}\ar@{.>}[r]<-2pt>_{}
  &\underline{B^{1}}-\mathbb{A}lg\ar[d]_{N_{\underline{B^{1}}}}\ar[r]<+2pt>^{\sigma^{1}_{0}}\ar[r]<-2pt>_{\beta^{1}_{0}}
  &\underline{B^{0}}-\mathbb{A}lg\ar[d]^{N_{\underline{B^{0}}}}\\
  \ar@{.>}[r]<+2pt>^{}\ar@{.>}[r]<-2pt>_{}&
  \widehat{\Theta_{\underline{B^{n}}}}\ar[r]<+2pt>^{\sigma^{n}_{n-1}}\ar[r]<-2pt>_{\beta^{n}_{n-1}}&  
  \widehat{\Theta_{\underline{B^{n-1}}}} \ar@{.>}[r]<+2pt>^{}\ar@{.>}[r]<-2pt>_{}&
    \widehat{\Theta_{\underline{B^{1}}}}\ar[r]<+2pt>^{\sigma^{1}_{0}}\ar[r]<-2pt>_{\beta^{1}_{0}}&  
\widehat{\Theta_{\underline{B^{0}}}}}\]  
  

  which finally achieve the goal of this paper by showing
   the following equivalence in $\mathbb{G}lob(\Cat)$ given by the Nerves Functors

  \[\xymatrix{\ar@{.>}[r]<+2pt>^{}\ar@{.>}[r]<-2pt>_{}
  &\underline{B^{n}}-\mathbb{A}lg\ar[d]_{N_{\underline{B^{n}}}}\ar[r]<+2pt>^{\sigma^{n}_{n-1}}\ar[r]<-2pt>_{\beta^{n}_{n-1}}
  &\underline{B^{n-1}}-\mathbb{A}lg\ar[d]^{N_{\underline{B^{n-1}}}}\ar@{.>}[r]<+2pt>^{}\ar@{.>}[r]<-2pt>_{}
  &\underline{B^{1}}-\mathbb{A}lg\ar[d]_{N_{\underline{B^{1}}}}\ar[r]<+2pt>^{\sigma^{1}_{0}}\ar[r]<-2pt>_{\beta^{1}_{0}}
  &\underline{B^{0}}-\mathbb{A}lg\ar[d]^{N_{\underline{B^{0}}}}\\
  \ar@{.>}[r]<+2pt>^{}\ar@{.>}[r]<-2pt>_{}&
  \Mod(\Theta_{\underline{B}^n})\ar[r]<+2pt>^{\sigma^{n}_{n-1}}\ar[r]<-2pt>_{\beta^{n}_{n-1}}&  
  \Mod(\Theta_{\underline{B}^{n-1}})\ar@{.>}[r]<+2pt>^{}\ar@{.>}[r]<-2pt>_{}&
    \Mod(\Theta_{\underline{B}^{1}})\ar[r]<+2pt>^{\sigma^{1}_{0}}\ar[r]<-2pt>_{\beta^{1}_{0}}&  
\Mod(\Theta_{\underline{B^{0}}})}\]

       \bigbreak{}
  \begin{minipage}{1.0\linewidth}
    Camell \textsc{Kachour}\\
    Macquarie University, Department of Mathematics\\
    Phone: 00612 9850 8942\\
    Email:\href{mailto:camell.kachour@mq.edu.au}{\url{camell.kachour@mq.edu.au}}
  \end{minipage}
  
\end{document}